\documentclass[a4paper,11pt]{article}
\usepackage{amsmath,amssymb,amsthm,graphicx}
\usepackage[top=80pt, bottom=90pt, left=70pt, right=70pt]{geometry}

\newtheorem{thm}{\bfseries Theorem}
\newtheorem{lem}[thm]{\bfseries Lemma}        
\newtheorem{prop}[thm]{\bfseries Proposition} 

\newtheorem{cl}[thm]{\bfseries Claim}

\usepackage{bm}
\usepackage{algorithm}
\usepackage{algorithmic}

\begin{document}
\title{On $k$-Submodular Relaxation\footnotemark[1]}
\author{Hiroshi HIRAI\footnotemark[2] \and Yuni IWAMASA\footnotemark[2]}
\footnotetext[1]{A preliminary version of this paper has appeared in the proceedings of the 9th Hungarian-Japanese Symposium on Discrete Mathematics and Its Applications.
This research was supported by JSPS KAKENHI Grant Numbers 25280004, 26280004, 26330023.
The second author was supported by JST, ERATO, Kawarabayashi Large Graph Project.}
\footnotetext[2]{Department of Mathematical Informatics,
Graduate School of Information Science and Technology,
University of Tokyo, Tokyo, 113-8656, Japan.\\
Email:\{hirai, yuni\_iwamasa\}@mist.i.u-tokyo.ac.jp
}
\maketitle


\begin{abstract}
	$k$-submodular functions, introduced by Huber and Kolmogorov, are functions defined on $\{0, 1, 2, \dots, k\}^n$ satisfying certain submodular-type inequalities.
	$k$-submodular functions typically arise as relaxations of NP-hard problems, and the relaxations by $k$-submodular functions play key roles in design of efficient, approximation, or fixed-parameter tractable algorithms.
	Motivated by this, we consider the following problem:
	Given a function $f : \{1, 2, \dots, k\}^n \rightarrow \mathbb{R} \cup \{+ \infty\}$, determine whether $f$ can be extended to a $k$-submodular function $g : \{0, 1, 2, \dots, k\}^n \rightarrow \mathbb{R} \cup \{+ \infty\}$, where $g$ is called a $k$-submodular relaxation of $f$, i.e., the restriction of $g$ on $\{1, 2, \dots, k\}^n$ is equal to $f$.
	
	We give a characterization, in terms of polymorphisms, of the functions which admit a $k$-submodular relaxation, and also give a combinatorial $O((k^n)^2)$-time algorithm to find a $k$-submodular relaxation or establish that a $k$-submodular relaxation does not exist.
	Our algorithm has interesting properties: (1) If the input function is integer valued, then our algorithm outputs a half-integral relaxation, and (2) if the input function is binary, then our algorithm outputs the unique optimal relaxation.
	We present applications of our algorithm to valued constraint satisfaction problems.
\end{abstract}
\begin{quote}
\end{quote}
\vspace{5mm}

\section{Introduction}
A {\em $k$-submodular function} (Huber and Kolmogorov \cite{ISCO/HK12}) is a function $f$ on $\{0, 1, 2, \dots, k\}^n$ satisfying the following inequalities,
\begin{align}
f(x) + f(y) \geq f(x \sqcap y) + f(x \sqcup y) \quad (x, y \in \{0, 1, 2, \dots, k\}^n),\label{ineq:ksub}
\end{align}
where binary operations $\sqcap, \sqcup$ are defined by
\[
(x \sqcap y)_i := \begin{cases} x_i & \text{if $x_i = y_i$,} \\ 0 & \text{if $x_i \neq y_i$,} \end{cases} \qquad (x \sqcup y)_i := \begin{cases} x_i & \text{if $x_i = y_i$,} \\ 0 & \text{if $0 \neq x_i \neq y_i \neq 0$,} \\ y_i & \text{if $x_i = 0$,}\\ x_i & \text{if $y_i = 0$} \end{cases}
\]
for $x = (x_1, x_2, \dots, x_n)$ and $y = (y_1, y_2, \dots, y_n)$.
Observe that $1$-submodular functions are submodular functions and $2$-submodular functions are bisubmodular functions (see \cite{book/Fujishige05}).

$k$-submodular functions typically arise as relaxations of NP-hard problems, and the relaxations by $k$-submodular functions, $k$-submodular relaxations, play key roles in the design of efficient, approximation, or fixed-parameter tractable (FPT) algorithms.
For a function $f$ on $\{1, 2, \dots, k\}^n$, a {\em $k$-submodular relaxation} \cite{ICCV/GK13, SICOMP/IWY16} of $f$ is a function $g$ on $\{0, 1, 2, \dots, k\}^n$ such that $g$ is $k$-submodular and the restriction of $g$ to $\{1, 2, \dots, k\}^n$ is equal to $f$.
Gridchyn and Kolmogorov \cite{ICCV/GK13} showed that the Potts energy function, a generalization of the objective of multiway cut, has a natural $k$-submodular relaxation, and that this relaxation is useful in computer vision applications.
Iwata, Wahlstr{\"o}m, and Yoshida \cite{SICOMP/IWY16} developed a general framework of FPT algorithms with introducing the concept of a {\em discrete relaxation}, where a $k$-submodular relaxation is a primary and important example of discrete relaxations.
Hirai \cite{DO/H15} introduced a class of discrete convex functions that can be locally relaxed to $k$-submodular functions, and designed efficient algorithms for some classes of multiflow and network design problems.

In view of these appearances and applications of $k$-submodular functions, it is quite natural and fundamental to consider the following problem:
Given a function $f$ on $\{1, 2, \dots, k\}^n$, determine whether there exists a $k$-submodular relaxation of $f$, and find a $k$-submodular relaxation if it exists.

The main results of this paper are a characterization of those functions which admit $k$-submodular relaxations, and a fast combinatorial algorithm to find a $k$-submodular relaxation.
Let $[k] := \{1, 2, \dots, k\}$ and $[0, k] := [k] \cup \{0\}$.
In this paper, functions can take the infinite value $+ \infty$, where $a < + \infty$ and $a + \infty = + \infty$ for $a \in \mathbb{R}$.
The $k$-submodular inequality (\ref{ineq:ksub}) is interpreted in this way.
(In the case of $f(x) = +\infty$ or $f(y) = +\infty$, (\ref{ineq:ksub}) trivially holds even if $f(x \sqcap y) = +\infty$ and $f(x \sqcup y) = +\infty$.)
Let $\overline{\mathbb{R}} := \mathbb{R} \cup \{+ \infty\}$.
For a function $f : D^n \rightarrow \overline{\mathbb{R}}$, let ${\rm dom}\ f := \{ x \in D^n \mid f(x) < + \infty \}$.
We show that the $k$-submodular extendability is characterized by a certain operation on $[k]$.
Let us define a ternary operation $\theta : [k]^3 \rightarrow [k]$ by
\[ \theta(a, b, c) := \begin{cases} a & \text{if $a = b$}, \\ c & \text{if $a \neq b$}. \end{cases} \]
The ternary operation $\theta$ is extended to a ternary operation $([k]^n)^3 \rightarrow [k]^n$ by $(\theta(x, y, z))_i = \theta(x_i, y_i, z_i)$.
Note that $\theta$ is a majority operation (in the sense of \cite{book/Zivny12}), since $\theta(a, a, b) = \theta(a, b, a) = \theta(b, a, a) = a$.
In universal algebra, $\theta$ is known as the {\it dual discriminator} \cite{AU/CHR99}; this fact was pointed out by A. Krokhin and the referees.

\begin{thm}\label{thm:chara}
A function $f : [k]^n \rightarrow \overline{\mathbb{R}}$ admits a $k$-submodular relaxation if and only if $\theta(x, y, z) \in {\rm dom}\ f$ for all $x, y, z \in {\rm dom}\ f$. 
\end{thm}
Therefore the class of $k$-submodular extendable functions is defined by a polymorphism $\theta$, and hence is closed under expressive power (see \cite{book/Zivny12}). Also the $k$-submodular extendability depends only on the domain of $f$, i.e., the labelings attaining finite value.
In particular, if ${\rm dom}\ f$ is the whole set $[k]^n$ (i.e., $f : [k]^n \rightarrow \mathbb{R}$), then $f$ always has a $k$-submodular relaxation; this fact has been noticed by Gridchyn and Kolmogorov \cite[p. 2325]{ICCV/GK13}, but their proof is not correct.\footnotemark[2]
\footnotetext[2]{They claimed that a $k$-submodular relaxation $g$ of arbitrary $f$ is obtained by setting $g(x) = f(x)$ for $x \in [k]^n$ and $g(x) = C$ for $x \in [0, k]^n \setminus [k]^n$, where $C \leq \min_{x \in [k]^n} f(x)$.
This is not true. Indeed, consider $f : [2]^2 \rightarrow \mathbb{R}$ such that $f(1, 2) := 1$ and $f(x) := 0$ for other $x \in [2]^2$.
Let $g : [0, 2]^2 \rightarrow \mathbb{R}$ be defined by $g(1,2) := 1$ and $g(x) := 0$ for other $x \in [0, 2]^2$.
Then $g$ is not $k$-submodular since $1 = g(0, 0) + g(1, 2) > g(1, 0) + g(0, 2) = 0$.}

Based on Theorem~\ref{thm:chara},
we can obviously test for the existence of a $k$-submodular relaxation in $O((k^n)^3)$ time by going through all labelings $x,y,z \in \textrm{dom }f$ and testing for the closure under $\theta$.
However this method cannot find a $k$-submodular relaxation even if it exists.
We will present a combinatorial  $O((k^n)^2)$-time algorithm to find a $k$-submodular relaxation.
Our algorithm reveals interesting and unexpected properties of the space of $k$-submodular relaxations:
the existence of a half-integral $k$-submodular relaxation and the existence of a unique maximal $k$-submodular relaxation in the case of $n = 2$.

\begin{thm}\label{thm:algo}
There exists an $O\left((k^n)^2\right)$-time algorithm to determine whether a function $f : [k]^n \rightarrow \overline{\mathbb{R}}$ has a $k$-submodular relaxation, and to construct a $k$-submodular relaxation $g$ if it exists, where $g$ has the following properties:
\begin{enumerate}
\item If $f$ is integer valued, then $g$ is half-integer-valued.
\item If $n = 2$, then for every $k$-submodular relaxation $g^\prime$ of $f$ it holds that
\[ g(x) \geq g^\prime(x) \quad (x \in {\rm dom}\ g).\]
Namely $g$ is the unique maximal $k$-submodular relaxation of $f$.
\end{enumerate}
\end{thm}
In particular, our algorithm outputs a half-integral and optimal $k$-submodular relaxation if $n = 2$. This solves, in the special case of binary $k$-submodular relaxations, a question raised by~\cite{SICOMP/IWY16}: {\rm Is there a way to decide the existence of discrete relaxations in general?}

The rest of this paper is organized as follows.
In Section \ref{sec:app}, we present applications of our algorithm to valued constraint satisfaction problems (VCSPs), where we utilize a recent remarkable result by Thapper and \v{Z}ivn\'{y} \cite{FOCS/TZ12} that $k$-submodular VCSPs can be solved in polynomial time. (The oracle tractability of $k$-submodular function minimization is one of the prominent open problems in the literature; see \cite{SIMDA/FT14, ISCO/HK12}.)
As a consequence of properties 1 and 2 in Theorem \ref{thm:algo}, our algorithm always constructs a half-integral $k$-submodular relaxation for integer-valued VCSPs, and an ``almost best'' $k$-submodular relaxation for binary VCSPs.
We also present an application to the maximization problem, where we utilize a recent result by Iwata, Tanigawa, and Yoshida \cite{SODA/ITY15} on $k$-submodular function maximization.
In Section \ref{sec:proof}, we prove Theorem \ref{thm:chara} and \ref{thm:algo}. Our algorithm is based on the Fourier--Motzkin elimination scheme for linear inequalities. We show that the system of $k$-submodular inequalities has a certain nice elimination ordering, and the Fourier--Motzkin elimination can be greedily carried out.

\section{Application}\label{sec:app}
Our algorithm is useful in {\em VCSPs}. Let us introduce VCSPs briefly; see \cite{book/Zivny12} for detail. Let $D$ be a finite set, called a {\em domain}.
By a {\em cost function} on $D$ we mean a function $f : D^r \rightarrow \overline{\mathbb{R}}$ for some natural number $r = r_f$, called the {\em arity} of $f$.
A set of cost functions is called a {\em language} on $D$.
For a language $\mathcal{L}$, a pair $(f, \sigma)$ of $f \in \mathcal{L}$ and $\sigma : \{1, 2, \dots, r_f\} \rightarrow \{1, 2, \dots, n\}$ is called a {\em constraint} on $\mathcal{L}$.
An {\em instance} of VCSP over language $\mathcal{L}$, denoted by ${\rm VCSP}(\mathcal{L})$, is a triple $I = (n, D, \mathcal{C})$ of the number $n$ of variables, domain $D$, and a finite set $\mathcal{C}$ of constraints on $\mathcal{L}$.
The task of ${\rm VCSP}(\mathcal{L})$ is to find $x = (x_1,\dots, x_n) \in D^n$ that minimizes
\[ f_I(x) := \sum_{(f, \sigma) \in \mathcal{C}} f(x_{\sigma(1)}, \dots, x_{\sigma(r_f)}). \]
Let ${\rm OPT}(I) := \min_x f_I(x)$.

In the case where $D = [0, k]$ and $\mathcal{L}$ consists of $k$-submodular functions, we call ${\rm VCSP}(\mathcal{L})$ a {\em $k$-submodular VCSP}. Thapper and \v{Z}ivn\'y \cite{FOCS/TZ12} proved the polynomial solvability of $k$-submodular VCSPs (see \cite{SICOMP/KTZ15} for the journal version).
\begin{thm}[see \cite{SICOMP/KTZ15, FOCS/TZ12}]\label{polytime}
$k$-submodular VCSPs can be solved in polynomial time.
\end{thm}
A {\em $k$-submodular relaxation} of an instance $I = (n, [k], \mathcal{C})$ is an instance $I^\prime = (n, [0, k], \mathcal{C}^\prime)$ such that $\mathcal{C}^\prime$ is obtained by replacing each cost function in $\mathcal{C}$ with its $k$-submodular relaxation. Notice that $f_{I^\prime}$ is a $k$-submodular relaxation of $f_I$.

\paragraph{$k$-submodular autarky.}
From a minimizer of a $k$-submodular relaxation, we obtain an {\em autarky}, a partial assignment of variables that keeps OPT, on the basis of the following property (called {\em persistency}).
\begin{thm}[see \cite{ICCV/GK13, SICOMP/IWY16}]\label{persistency}
Let $f$ be a function on $[k]^n$ and $g$ a $k$-submodular relaxation of $f$.
For any minimizer $y^* \in [0,k]^n$ of $g$,
there exists a minimizer $x^* \in [k]^n$ of $f$
such that $x^*_i = y^*_i$ for all $i$ with $y^*_i \neq 0$. 
\end{thm}
By our algorithm, for an instance $I = (n, [k], \mathcal{C})$, we can construct a $k$-submodular relaxation $I^\prime$, if it exists, in $O(|\mathcal{C}|(k^r)^2)$ time, where $r$ is the maximum arity of a function in $\mathcal{C}$. This is a polynomial time algorithm in VCSPs.
By Theorem \ref{polytime}, we obtain an optimal solution $y^*$ of $I^\prime$ in polynomial time.
By Theorem \ref{persistency}, in solving $I$, we can fix $x_i$ to $y^*_i$ for all $i$ with $y^*_i \neq 0$. This contributes to reducing the size of VCSP.

\paragraph{FPT algorithm.}
Iwata, Wahlstr{\"o}m, and Yoshida \cite{SICOMP/IWY16} present an application of $k$-submodular relaxation for FPT algorithms.
Suppose that $\mathcal{L}$ consists of integer-valued cost functions.
For an instance $I = (n, [k], \mathcal{C})$ of ${\rm VCSP}(\mathcal{L})$ and a $k$-submodular relaxation $I^\prime = (n, [0, k], \mathcal{C}^\prime)$ of $I$, the {\em scaling factor} of $I^\prime$ is the smallest integer $c$ such that $c\cdot g$ is integer valued for all $g \in \mathcal{C}^\prime$.
Let $d := {\rm OPT}(I) - {\rm OPT}(I^\prime)$.
Then we can solve an instance $I$ in polynomial time, provided $k^{cd}$ is fixed.
\begin{thm}[{see\footnotemark[3] \cite[Lemma 1]{SICOMP/IWY16}}]\label{FPT}
Let $\mathcal{L}$ be a language on $[k]$ consisting of integer-valued cost functions.
Let $I$ be an instance of ${\rm VCSP}(\mathcal{L})$, and $I^\prime$ a $k$-submodular relaxation of $I$ with the scaling factor $c$.
Let $d := {\rm OPT}(I) - {\rm OPT}(I^\prime)$.
We can solve the instance $I$ by solving a $k$-submodular VCSP at most $k^{cd}$ times.
\end{thm}
\footnotetext[3]{Note that our definition of $k$-submodular relaxation is slightly different from the one given by \cite{SICOMP/IWY16}, where the definition in \cite{SICOMP/IWY16} requires one more condition $\min g = \min f$. Theorem \ref{FPT} holds in our setting since the proof does not use this condition.}
Our algorithm constructs a $k$-submodular relaxation $I'$ with $c = 1$ or $2$, though we do not say anything about the magnitude of $d$ in general.
In the binary case ($r = 2$) that includes many important VCSPs, our relaxation is an almost best $k$-submodular relaxation in the following sense:
Our relaxation has the smallest $d$ among all $k$-submodular relaxations.
Indeed, for any $k$-submodular relaxation $I^{\prime\prime}$, it holds that $f_{I^\prime}(x) \geq f_{I^{\prime\prime}}(x)$ by property 2 in Theorem~\ref{thm:algo}, and hence $d \leq {\rm OPT}(I) - {\rm OPT}(I^{\prime\prime})$.
Moreover our relaxation $I'$ has the smallest $cd$, except for the case where there exists another $k$-submodular relaxation $I''$ with scaling factor 1 and ${\rm OPT}(I)-{\rm OPT}(I'') < cd$.
In such a case, the obtained $cd$ is still a 2-approximation.

In FPT applications, $d^* = {\rm OPT}(I)$ is a more desirable parameter than $d = {\rm OPT}(I) - {\rm OPT}(I')$, since $d^*$ depends only on input $I$ (see \cite{SICOMP/IWY16}).
If ${\rm OPT}(I') \geq 0$, then $d \leq d^*$, and we can use $d^*$ as an FPT parameter.
This is in the case where each function in $I'$ is nonnegative valued.
This leads to the notion of a {\it nonnegative(-valued) $k$-submodular relaxation}.
For binary functions, our algorithm returns a nonnegative $k$-submodular relaxation if it exists.
This fact will be useful in design of FPT algorithms for binary VCSPs.
It should be noted that $k$-submodular relaxations for special binary functions (given in \cite[Section 4.1]{SICOMP/IWY16}), are the same as relaxations obtained by our algorithm.

\paragraph{Maximization.}
Nonnegative $k$-submodular relaxation also has a potential to provide a unified approach to maximization.
We here consider maximization of functions having no $+\infty$.
Following \cite{METR/ITY13,TALG/WZ15}, just recently, Iwata, Tanigawa, and Yoshida~\cite{SODA/ITY15} presented a $1/2$-approximation algorithm for nonnegative $k$-submodular function maximization.
\begin{thm}[see {\cite[Theorem 2.3]{SODA/ITY15}}]\label{Max}
There exists a polynomial time randomized $1/2$-approximation algorithm for maximizing nonnegative $k$-submodular functions (given by value oracle).
\end{thm}
The maximum of a $k$-submodular function $g$ on $[0, k]^n$ is always attained at $[k]^n$ (see \cite[Proposition 2.1]{SODA/ITY15}).
Indeed, for any maximizer $z \in [0, k]^n$, choose any $z_1, z_2 \in [k]^n$ with $z = z_1 \sqcap z_2 = z_1 \sqcup z_2$; 
both $z_1$ and $z_2$ are maximizers by $k$-submodularity, $g(z_1) + g(z_2) \geq 2g(z)$.
Therefore if we have a nonnegative $k$-submodular relaxation of given $f$ and its value oracle, then by using Iwata--Tanigawa--Yoshida algorithm we obtain an approximate maximum solution of $f$ with factor $1/2$ on average.
Our algorithm is again useful for the case where $f$ is given as a VCSP form, i.e., the sum of small arity functions.
If we obtain a nonnegative relaxation for each summand, then this approximation scheme is applicable.

Further study on nonnegative $k$-submodular relaxation is left to future work.

\section{Proofs}\label{sec:proof}
For $x \in [0,k]^n$, let $Z(x)$ denote the number of indices $i$ with $x_i = 0$. For $A \subseteq [k]^n$, let $C_{\theta}(A)$ denote the minimum subset $X$ of $[k]^n$ containing $A$ such that $\theta(x, y, z) \in X$ for all $x, y, z \in X$. For $B \subseteq [0, k]^n$, let $C_{\sqcap}(B)$ (resp., $C_{\sqcap, \sqcup}(B)$) denote the minimum subset $X$ of $[0, k]^n$ containing $B$ such that $x \sqcap y \in X$ (resp., $x \sqcap y, x \sqcup y \in X$) for all $x, y \in X$. Note that all sets, $C_{\theta}(\cdot), C_{\sqcap}(\cdot), C_{\sqcap, \sqcup}(\cdot)$, are uniquely determined. In particular, $A \mapsto C_{\theta}(A)$, $B \mapsto C_{\sqcap}(B)$ and $B \mapsto C_{\sqcap, \sqcup}(B)$ are closure operators. Observe that $\theta$ can be represented by $\sqcup$ as follows:
\begin{align}\label{eq:theta}
\theta(x, y, z) = ((x \sqcup y) \sqcup z) \sqcup (x \sqcup y).
\end{align}

\subsection{Proof of Theorem \ref{thm:chara}}
\begin{lem}\label{lem:sqcap}
For all $A \subseteq [k]^n$, it holds that $C_{\sqcap, \sqcup}(C_{\theta}(A)) = C_{\sqcap}(C_{\theta}(A))$.
\end{lem}
\begin{proof}
The inclusion ($\supseteq$) is obvious. Therefore it suffices to prove that $C_{\sqcap}(C_{\theta}(A))$ is closed under $\sqcup$. Take arbitrary $x, y \in C_{\sqcap}(C_{\theta}(A))$. Our goal is to show $x \sqcup y \in C_{\sqcap}(C_{\theta}(A))$. By the definition of $C_{\sqcap}$, there are $x^1, \dots, x^s, y^1, \dots, y^t \in C_{\theta}(A)$ such that $x = x^1 \sqcap \dots \sqcap x^s$ and $y = y^1 \sqcap \dots \sqcap y^t$ (note that $\sqcap$ is associative).
Let $X := \{ i \mid x_i \neq 0 \}$ and $Y := \{ j \mid y_j \neq 0 \}$. First we show that there exists $u \in C_{\theta}(A)$ such that $u_i = x_i$ for $i \in X$ and $u_i = y_i$ for $i \in Y \setminus X$. By the definition of $\sqcup$, $X$, and $Y$, we have
\[ (x \sqcup y)_i = \begin{cases} 0 & \text{if $x_i = y_i = 0$ or $0 \neq x_i \neq y_i \neq 0$},\\ x_i & \text{if $i \in X \setminus Y$ or $x_i = y_i \neq 0$,} \\ y_i & \text{if $i \in Y \setminus X$}. \end{cases} \]
Let $Y \setminus X = \{j_1,\dots,j_a\}$. For all $j \in Y \setminus X$, there exists a pair $(p_j,q_j)$ of indices in $\{1,2,..,s\}$ such that $x^{p_j}_j \neq x^{q_j}_j$, since $x_j = 0$. Let $u^j := \theta(x^{p_j}, x^{q_j}, y^1)$. Then we have $u^j_i = x^{p_j}_i = x^{q_j}_i = x_i$ for $i \in X$, and $u^j_j = y^1_j = y_j$. Define $u^{j_1\dots j_k} := \theta(u^{j_1 \dots j_{k-1}}, u^{j_k}, y^1)$ $(2 \leq k \leq a)$. It is easily seen that $u^{j_1\dots j_a} \in C_{\theta}(A)$, $u^{j_1\dots j_a}_i = x_i$ for $i \in X$, and $u^{j_1\dots j_a}_i = y_i$ for $i \in Y \setminus X$. Similarly, there exists $v \in C_{\theta}(A)$ such that $v_i = y_i$ for $i \in Y$ and $v_i = x_i$ for $i \in X \setminus Y$. Hence $u \sqcap v \in C_{\sqcap}(C_{\theta}(A))$. It holds that $(u \sqcap v)_i = (x \sqcup y)_i$ for all $i \in X \cup Y$. Therefore the set $B := \{ z \in C_{\sqcap}(C_{\theta}(A)) \mid \text{$z_i = (x \sqcup y)_i$ for all $i \in X \cup Y$} \}$ is nonempty.

Take $z \in B$ with maximum $Z(z)$. We show $z = x \sqcup y$ (implying $x \sqcup y \in C_{\sqcap}(C_{\theta}(A))$, as required). Suppose to the contrary that $z \neq x \sqcup y$. By assumption,  there exists $l$ such that $z_l \neq x_l = y_l = 0$. For this $l$, there exist $p, q, r$ such that $x^p_l \neq x^q_l$ and $y^r_l \neq z_l$, since $x_l = y_l = 0$. Let $w^0 := \theta(x^p, x^q, y^r)$. It holds that $w^0_i = x_i$ for $i \in X$, and $w^0_l = y^r_l \neq z_l$. Define $w := \theta(w^0, u, y^r)$ ($u \in C_{\theta}(A)$ such that $u_i = x_i$ for $i \in X$ and $u_i = y_i$ for $i \in Y \setminus X$). It is clear that $w_i = x_i$ for $i \in X$, $w_i = y_i$ for $i \in Y \setminus X$, and $w_l = y^r_l \neq z_l$. Therefore $z \sqcap w \in B$ and $Z(z) < Z(z \sqcap w)$, since $(z \sqcap w)_l = 0$. However this is a contradiction to the maximality of $z$. Thus $z = x \sqcup y$.
\end{proof}

\begin{lem}\label{lem:theta}
For all $A \subseteq [k]^n$, it holds that $C_{\sqcap, \sqcup}(A) \cap [k]^n = C_{\theta}(A)$.
\end{lem}
\begin{proof}
By (\ref{eq:theta}), we have $C_{\sqcap, \sqcup}(A) \supseteq C_{\theta}(A) \supseteq A$. Since $C_{\sqcap, \sqcup}$ is a closure operator, it holds that $C_{\sqcap, \sqcup}(A) = C_{\sqcap, \sqcup}(C_{\sqcap, \sqcup}(A)) \supseteq C_{\sqcap, \sqcup}(C_{\theta}(A)) \supseteq C_{\sqcap, \sqcup}(A)$. In particular, $C_{\sqcap, \sqcup}(C_{\theta}(A)) = C_{\sqcap, \sqcup}(A)$. Hence $C_{\sqcap, \sqcup}(A) \cap [k]^n = C_{\sqcap, \sqcup}(C_{\theta}(A)) \cap [k]^n$. By Lemma \ref{lem:sqcap}, $C_{\sqcap, \sqcup}(C_{\theta}(A)) \cap [k]^n = C_{\sqcap}(C_{\theta}(A)) \cap [k]^n$. 
Here $C_{\sqcap}(C_{\theta}(A)) \cap [k]^n = C_{\theta}(A)$ holds.
Indeed, for $x \in C_{\sqcap}(C_{\theta}(A)) \cap [k]^n$, there are $x^1, x^2, \dots, x^m \in C_{\theta}(A)$ with $x = x^1 \sqcap \dots \sqcap x^m$. Then $x^1 = x^2 = \dots = x^m = x$ must hold, otherwise $x$ includes $0$, contradicting $x \in [k]^n$.
Thus $x \in C_{\theta}(A)$ and $C_{\sqcap}(C_{\theta}(A)) \cap [k]^n = C_{\theta}(A)$. Consequently, $C_{\sqcap, \sqcup}(A) \cap [k]^n = C_{\theta}(A)$.
\end{proof}

\begin{prop}\label{prop:chara}
A function $f : [k]^n \rightarrow \overline{\mathbb{R}}$ admits a $k$-submodular relaxation if and only if
\[C_{\sqcap, \sqcup}({\rm dom}\ f) \cap [k]^n = {\rm dom}\ f.\]
\end{prop}
\begin{proof}
(only-if part). 
Suppose that $f$ has a $k$-submodular relaxation $g$. By the equation ${\rm dom}\ f = {\rm dom}\ g \cap [k]^n$, we have
\begin{align*}
C_{\sqcap, \sqcup}({\rm dom}\ f) \cap [k]^n \supseteq {\rm dom}\ f = {\rm dom}\ g \cap [k]^n = C_{\sqcap, \sqcup}({\rm dom}\ g) \cap [k]^n\\
\supseteq (C_{\sqcap, \sqcup}({\rm dom}\ g \cap [k]^n)) \cap [k]^n = C_{\sqcap, \sqcup}({\rm dom}\ f) \cap [k]^n,
\end{align*}
where ${\rm dom}\ g = C_{\sqcap, \sqcup}({\rm dom}\ g)$ follows from the fact that the domain of a $k$-submodular function is closed under $\sqcap, \sqcup$.
Thus $C_{\sqcap, \sqcup}({\rm dom}\ f) \cap [k]^n = {\rm dom}\ f$, as required.

(if part). Assume $C_{\sqcap, \sqcup}({\rm dom}\ f) \cap [k]^n = {\rm dom}\ f$. Let $N := |C_{\sqcap, \sqcup}({\rm dom}\ f)|$ and $N^\prime := |{\rm dom}\ f|$. We can consider a function $g : [0, k]^n \rightarrow \overline{\mathbb{R}}$ with ${\rm dom}\ g = C_{\sqcap, \sqcup}({\rm dom}\ f)$ as a vector $\bm{g} \in \mathbb{R}^N$, where the $x$th component $\bm{g}_x$ of $\bm{g}$ for $x \in {\rm dom}\ g$ is defined as $g(x)$. 
By definition, the set of all $k$-submodular functions is defined by linear inequalities (\ref{ineq:ksub}), and hence forms a polyhedron $P = \{ \bm{g} \in \mathbb{R}^N \mid A\bm{g} \leq \bm{0}\}$.
Therefore, the set of functions which admit a $k$-submodular relaxation can be considered as the projection $P^\prime := \{ \bm{f} \in \mathbb{R}^{N^\prime} \mid \bm{g} \in P,\ \bm{f} \text{ is the projection of $\bm{g}$ to $\mathbb{R}^{N^\prime}$} \}$ of $P$. Let us prove that the projection $P^\prime$ is equal to $\mathbb{R}^{N^\prime}$.

We can obtain $P^\prime$ by the elimination of all variables $\bm{g}_x$ ($x \in {\rm dom}\ g \setminus {\rm dom}\ f$) by using the Fourier--Motzkin elimination method.
We repeatedly eliminate variables $\bm{g}_x$ by taking an index $x$ with maximum $Z(x)$, i.e., the number of zeros in $x$, in each step. Suppose that an index $x$ is chosen in the first step.
Then coefficients of $\bm{g}_x$ in $A\bm{g} \leq 0$ are positive or zero. Indeed, assume that there exists an inequality such that the coefficient of $\bm{g}_x$ is negative. Namely, there is a nontrivial inequality
\begin{align}\label{ineq:minus}
g(x \sqcap y) + g(x \sqcup y) - g(x) - g(y) \leq 0
\end{align}
with $x \sqcap y \neq x \neq x \sqcup y$.
By the definition of $x$, we have $Z(x) \geq Z(x \sqcap y)$. By the definition of $\sqcap$, we have $Z(x) \leq Z(x \sqcap y)$. Thus $Z(x) = Z(x \sqcap y)$, $x = x \sqcap y$, and $y = x \sqcup  y$. This means that the coefficient of $\bm{g}_x$ is positive in all inequalities containing $\bm{g}_x$. So the linear inequality system of the projection is obtained by simply removing all inequalities containing variable $\bm{g}_x$.
Now suppose that the set $S$ of variables has been eliminated by the Fourier--Motzkin method, and the corresponding system of inequalities consists of the original inequalities not containing variables in $S$, as above.
In the next step, the Fourier--Motzkin procedure chooses an index $x$ with maximum $Z(x)$ over ${\rm dom}\ g \setminus S$. By a similar argument, the coefficients of $\bm{g}_x$ in the current inequalities are positive or zero; inequalities having negative coefficients at $\bm{g}_x$ have already been removed in the previous steps.
Thus there are finally no inequalities. (Note that there exist no inequalities only containing elements in ${\rm dom}\ f$, since if $x, y \in {\rm dom}\ f$ and $x \neq y$, then $x \sqcap y = x \sqcup y \in {\rm dom}\ g \setminus {\rm dom}\ f$.)
This means that $P^\prime = \mathbb{R}^{N^\prime}$, as required.
\end{proof}

By the proof of Proposition \ref{prop:chara}, the $k$-submodular inequalities can be represented as follows: For any $x, y \in {\rm dom}\ g$ such that $z = x \sqcap y$ and $Z(z) > \max\{Z(x), Z(y)\}$,
\begin{align}\label{ineq:z}
	g(z) \leq \begin{cases}\displaystyle \frac{1}{2} \left(g(x) + g(y)\right) & \text{if $x \sqcap y = x \sqcup y$},\\
		g(x) + g(y) - g(x \sqcup y) & \text{otherwise}. \end{cases}
\end{align}

We are now ready to prove Theorem \ref{thm:chara}.
By Lemma \ref{lem:theta} and Proposition \ref{prop:chara}, a function $f : [k]^n \rightarrow \overline{\mathbb{R}}$ admits a $k$-submodular relaxation if and only if $C_{\theta}({\rm dom}\ f) = {\rm dom}\ f$. It is clear that this statement is the same as Theorem \ref{thm:chara}.

\subsection{Proof of Theorem \ref{thm:algo}}
We present an algorithm with the claimed properties in Algorithm \ref{algo}.
\begin{algorithm}
\caption{$k$-submodular relaxation}\label{algo}
\label{alg1}
\begin{algorithmic}
\STATE Initialize $g^{(0)}$ as follows:
\[ g^{(0)}(x) := \begin{cases} f(x) & \text{if $x \in {\rm dom}\ f$}, \\ + \infty & {\rm otherwise}. \end{cases} \]
\FOR{$i = 1$ to $n$}
\STATE $g^{(i)} := g^{(i-1)}$
\FOR{all $x, y \in {\rm dom}\ g^{(i-1)}$ such that $Z(x \sqcap y) = i$}
\IF{$x \sqcap y = x \sqcup y$}
\IF{$g^{(i)}(x \sqcap y) > \left( g^{(i-1)}(x) + g^{(i-1)}(y)\right)/2$}
\STATE $g^{(i)}(x \sqcap y) := \left(g^{(i-1)}(x) + g^{(i-1)}(y)\right)/2$
\ENDIF
\ELSIF{$x \sqcup y \not\in {\rm dom}\ g^{(i-1)}$}
\RETURN $f$ has no $k$-submodular relaxation.
\ELSE
\IF{$g^{(i)}(x \sqcap y) > g^{(i-1)}(x) + g^{(i-1)}(y) - g^{(i-1)}(x \sqcup y)$}
\STATE $g^{(i)}(x \sqcap y) := g^{(i-1)}(x) + g^{(i-1)}(y) - g^{(i-1)}(x \sqcup y)$
\ENDIF
\ENDIF
\ENDFOR
\ENDFOR
\RETURN $g^{(n)}$
\end{algorithmic}
\end{algorithm}
Let us prove that Algorithm~\ref{algo} correctly determines whether a function $f : [k]^n \rightarrow \overline{\mathbb{R}}$ has a $k$-submodular relaxation, and constructs a $k$-submodular relaxation if it exists. 
First we show that if Algorithm \ref{algo} returns ``$f$ has no $k$-submodular relaxation," then the input function $f$ actually has no $k$-submodular relaxation. 
To prove this statement, we show the following claim.
\begin{cl}\label{cl}
It holds that ${\rm dom}\ g^{(i)} = \{ z \in C_{\sqcap}({\rm dom}\ f) \mid Z(z) \leq i \}$.
\end{cl}
\begin{proof}[Proof of Claim \ref{cl}]
The inclusion ($\subseteq$) is obvious. We prove ($\supseteq$) by induction on $i$. The case $i = 1$ is trivial.
Indeed, any element $z$ of ${\rm dom}\ g = C_{\sqcap}({\rm dom}\ f)$ with $Z(z) = 1$ can be written as $z = x \sqcap y$ with $x, y \in {\rm dom}\ f$.
For all $z \in C_{\sqcap}({\rm dom}\ f) \setminus {\rm dom}\ f$ with $Z(z) = i+1$, there exist $x, y \in C_{\sqcap}({\rm dom}\ f)$ such that $\max\{Z(x), Z(y)\} \leq i$ and $z = x \sqcap y$. By the induction hypothesis, we have $x, y \in {\rm dom}\ g^{(i)}$, and hence $z \in {\rm dom}\ g^{(i+1)}$. This completes the induction step.
\end{proof}
Here we consider the case of returning ``$f$ has no $k$-submodular relaxation.''
In this case, for some step $i > 1$,  there are $x, y \in {\rm dom}\ g^{(i-1)}$ such that $Z(x \sqcap y) = i$, $x \sqcap y \neq x \sqcup y$, and $x \sqcup y \not\in {\rm dom}\ g^{(i-1)}$.
Then $x, y \in {\rm dom}\ g^{(i-1)}$ and $Z(x \sqcap y) > Z(x \sqcup y)$.
Therefore $x \sqcup y \not\in {\rm dom}\ g^{(i-1)} = \{ z \in C_{\sqcap}({\rm dom}\ f) \mid Z(z) \leq i-1 \}$ (by Claim \ref{cl}).
On the other hand, it is obvious that $x \sqcup y \in \{ z \in C_{\sqcap, \sqcup}({\rm dom}\ f) \mid Z(z) \leq i-1\}$.
This means that $C_{\sqcap}({\rm dom}\ f) \neq C_{\sqcap, \sqcup}({\rm dom}\ f)$.
By Lemma \ref{lem:sqcap}, it necessarily holds that $C_{\theta}({\rm dom}\ f) \neq {\rm dom}\ f$.
By Theorem \ref{thm:chara}, there is no $k$-submodular relaxation for $f$.

Next we consider the case of returning $g^{(n)}$. Returning $g^{(n)}$ means $C_{\sqcap, \sqcup}({\rm dom}\ f) = C_{\sqcap}({\rm dom}\ f)$ since for all $x, y \in C_{\sqcap}({\rm dom}\ f)$, $x \sqcup y \in C_{\sqcap}({\rm dom}\ f)$. Therefore $C_{\sqcap, \sqcup}({\rm dom}\ f) \cap [k]^n = {\rm dom}\ f$, and $f$ admits a $k$-submodular relaxation by Proposition \ref{prop:chara}. 
Furthermore $g^{(n)}$ defined by Algorithm 1 satisfies (\ref{ineq:z}). Thus $g^{(n)}$ is a $k$-submodular function.

Next we show the half-integrality property (property 1 in Theorem \ref{thm:algo}).
\begin{prop}\label{prop:half}
Suppose that a function $f : [k]^n \rightarrow \overline{\mathbb{R}}$ has a $k$-submodular relaxation. Let $g$ be a $k$-submodular relaxation of $f$ constructed by Algorithm \ref{algo}. For all $z \in {\rm dom}\ g \setminus {\rm dom}\ f$, there exist $x, y \in {\rm dom}\ g$ with $z = x \sqcap y$ and $Z(z) > \max\{Z(x), Z(y)\}$ satisfying 1 or 2:
\begin{enumerate}
\item $\displaystyle x, y \in {\rm dom}\ f \text{ and } g(z) = \frac{1}{2} (g(x) + g(y))$;
\item $x \sqcap y \neq x \sqcup y \text{ and } g(z) = g(x) + g(y) - g(x \sqcup y)$.
\end{enumerate}
In particular, if $f$ is integer valued, then $g$ is half-integer-valued
\end{prop}
\begin{proof}
We will prove this by induction on $Z(z)$. 
By Algorithm~\ref{algo}, for some $x, y \in {\rm dom}\ g$ with $z = x \sqcap y$ and $Z(z) > \max\{Z(x), Z(y)\}$, the value $g(z)$ is equal to $(g(x) + g(y))/2$ if $x \sqcap y = x \sqcup y$, and $g(x) + g(y) - g(x \sqcup y)$ otherwise.
If $Z(z) = 1$, then $Z(x) = Z(y) = 0$ implying $x, y \in {\rm dom}\ f$, $x \sqcap y = x \sqcup y$, and $g(z) = (g(x) + g(y))/2$; we are in 1.
Suppose that $Z(z) \geq 2$, and that $x, y$ satisfy neither 1 nor 2.
Then $x \sqcap y = x \sqcup y$, $g(z) = (g(x) + g(y))/2$, and therefore one of $x, y$ does not belong to ${\rm dom}\ f$.
Say $x \not\in {\rm dom}\ f$.
Here $1 \leq Z(x) < Z(z) = Z(x \sqcap y)$. So by the induction hypothesis applied to $x$, we only consider two cases:
\begin{description}
\item[Case 1:] $\displaystyle g(x) = \frac{1}{2} (g(x^1) + g(x^2))$ for some $x^1, x^2 \in {\rm dom}\ f$ with $x^1 \sqcap x^2 = x$;
\item[Case 2:] $g(x) = g(x^1) + g(x^2) - g(x^1 \sqcup x^2)$ for some $x^1, x^2 \in {\rm dom}\ g$ with $x^1 \sqcup x^2 \neq x^1 \sqcap x^2 = x$ and $Z(x) > \max\{Z(x^1), Z(x^2)\}$.
\end{description}
Let $\mathcal{Z} := \{ i \mid z_i \neq 0 \}$. Let $\mathcal{X} := \{ i \not\in \mathcal{Z} \mid x_i \neq 0 \}$ and $\mathcal{Y} := \{ i \not\in \mathcal{Z} \mid y_i \neq 0 \}$. Note that $\mathcal{X} = \mathcal{Y}$, $0 \neq x_i \neq y_i \neq 0$ for $i \in \mathcal{X}$, and $x_i = y_i \neq 0$ for $i \in \mathcal{Z}$, since $z = x \sqcap y = x \sqcup y$. Hence $Z(x) = Z(y) \geq 1$, i.e., $y \in {\rm dom}\ g \setminus {\rm dom}\ f$.

In Case 1, it holds that $x_i = x^1_i = x^2_i$ for $i \in \mathcal{X} \cup \mathcal{Z}$ and $x^1_i \neq x^2_i$ for $i \not\in \mathcal{X} \cup \mathcal{Z}$. We obtain
\begin{align}
g(z) &= \frac{1}{2} (g(x) + g(y))\label{c1:1} \\
&= \frac{1}{2} \left( \left(\frac{1}{2} g(x^1) + \frac{1}{2} g(x^2)\right) + g(y) \right)\label{c1:2} \\
&\geq \frac{1}{2} \left( \frac{1}{2} g(x^1 \sqcap y) + \frac{1}{2} g(x^1 \sqcup y) + \frac{1}{2} g(x^2 \sqcap y) + \frac{1}{2} g(x^2 \sqcup y) \right)\label{c1:3} \\
&= \frac{1}{2} \left( g(z) + \frac{1}{2} g(x^1 \sqcup y) + \frac{1}{2} g(x^2 \sqcup y) \right)\label{c1:4} \\
&\geq \frac{1}{2} (g(z) + g(z)) = g(z).\label{c1:5}
\end{align}
Indeed, (\ref{c1:1}) $=$ (\ref{c1:2}) follows from the assumption of Case 1, and (\ref{c1:2}) $\geq$ (\ref{c1:3}) follows from the $k$-submodularity ($g(x^1) + g(y) \geq g(x^1 \sqcap y) + g(x^1 \sqcup y)$ and $g(x^2) + g(y) \geq g(x^2 \sqcap y) + g(x^2 \sqcup y)$).
Since $x_i = x^1_i = x^2_i = y_i$ for $i \in \mathcal{Z}$, $x_i = x^1_i = x^2_i \neq y_i$ for $i \in \mathcal{Y}$, and $y_i = 0$ for $i \not\in \mathcal{Y} \cup \mathcal{Z}$, it holds that $x^1 \sqcap y = x^2 \sqcap y = z$. Hence (\ref{c1:3}) $=$ (\ref{c1:4}). Since $(x^1 \sqcup y)_i = (x^2 \sqcup y)_i$ for $i \in \mathcal{Z}$, $(x^1 \sqcup y)_i = (x^2 \sqcup y)_i = 0$ for $i \in \mathcal{X}$, and $x^1_i = (x^1 \sqcup y)_i \neq (x^2 \sqcup y)_i = x^2_i$ for $i \not\in \mathcal{X} \cup \mathcal{Z}$, it holds that $(x^1 \sqcup y) \sqcap (x^2 \sqcup y) = (x^1 \sqcup y) \sqcup (x^2 \sqcup y) = z$. Hence (\ref{c1:4}) $\geq$ (\ref{c1:5}) follows from the $k$-submodularity.
This means that all inequalities are equalities.
Therefore $g(x^1) + g(y) = g(x^1 \sqcap y) + g(x^1 \sqcup y)$ by (\ref{c1:2}) = (\ref{c1:3}). (It is also true that $g(x^2) + g(y) = g(x^2 \sqcap y) + g(x^2 \sqcup y)$.)
Here $g(x^1 \sqcap y) = g(z)$.
Thus $g(z) = g(x^1) + g(y) - g(x^1 \sqcup y)$.
Moreover it holds that $z \neq x^1 \sqcup y$.
Indeed, for $i \not\in \mathcal{Y} \cup \mathcal{Z}$ we have $z_i = (x^1 \sqcap y)_i = 0$ and $(x^1 \sqcup y)_i = x^1_i \neq 0$, since $y_i = 0$.
So $g(z) = g(x^1) + g(y) - g(x^1 \sqcup y)$ means that $x^1, y$ satisfy 2.

In Case 2, it holds that $x_i = x^1_i = x^2_i$ for $i \in \mathcal{X} \cup \mathcal{Z}$ and $(x^1 \sqcap x^2)_i = 0$ for $i \not\in \mathcal{X} \cup \mathcal{Z}$.
We obtain
\begin{align}
g(z) &= \frac{1}{2} (g(x) + g(y))\label{c2:1} \\
&= \frac{1}{2} (g(x^1) + g(x^2) - g(x^1 \sqcup x^2) + g(y))\label{c2:2} \\
&\geq \frac{1}{2} (g(x^1 \sqcap y) + g(x^1 \sqcup y) + g(x^2) - g(x^1 \sqcup x^2))\label{c2:3} \\
&\geq \frac{1}{2} (g(z) + g((x^1 \sqcup y) \sqcap x^2) + g((x^1 \sqcup y) \sqcup x^2) - g(x^1 \sqcup x^2))\label{c2:4} \\
&= \frac{1}{2} (g(z) + g(z) + g(x^1 \sqcup x^2) - g(x^1 \sqcup x^2)) = g(z).\label{c2:5}
\end{align}
Indeed, (\ref{c2:1}) $=$ (\ref{c2:2}) follows from the assumption of Case 2, and (\ref{c2:2}) $\geq$ (\ref{c2:3}) follows from the $k$-submodularity.
Since $x_i = x^1_i = y_i$ for $i \in \mathcal{Z}$, $x_i = x^1_i \neq y_i$ for $i \in \mathcal{X}$, and $y_i = 0$ for $i \not\in \mathcal{X} \cup \mathcal{Z}$, it holds that $x^1 \sqcap y = z$. Hence (\ref{c2:3}) $\geq$ (\ref{c2:4}) follows from the $k$-submodularity. Since $(x^1 \sqcup y)_i = x^2_i$ for $i \in \mathcal{Z}$, $(x^1 \sqcup y)_i = 0$ for $i \in \mathcal{X}$, and $(x^1 \sqcup y)_i = x^1_i$ for $i \not\in \mathcal{X} \cup \mathcal{Z}$, it holds that $(x^1 \sqcup y) \sqcap x^2 = z$ and $(x^1 \sqcup y) \sqcup x^2 = x^1 \sqcup x^2$. Hence (\ref{c2:4}) $=$ (\ref{c2:5}).
This means that all inequalities are equalities. Therefore $g(x^1) + g(y) = g(x^1 \sqcap y) + g(x^1 \sqcup y)$ by (\ref{c2:2}) = (\ref{c2:3}). Here $g(x^1 \sqcap y) = g(z)$.
Thus $g(z) = g(x^1) + g(y) - g(x^1 \sqcup y)$.
Moreover it holds that $z \neq x^1 \sqcup y$.
Indeed, there exists $i \not\in \mathcal{Y} \cup \mathcal{Z}$ such that $x^1_i \neq 0$, since $Z(y) = Z(x) > Z(x^1)$.
Hence $z_i = (x^1 \sqcap y)_i = 0$ and $(x^1 \sqcup y)_i = x^1_i \neq 0$.
So $g(z) = g(x^1) + g(y) - g(x^1 \sqcup y)$ means that $x^1, y$ satisfy 2.

Next let us prove that if $f$ is integer valued, then for all $z \in {\rm dom}\ g \setminus {\rm dom}\ f$, $g(z)$ is half-integral.
The proof is by induction on $Z(z)$. The case $Z(z) = 1$ is trivial.
Suppose that $Z(z) \geq 2$.
For all $z \in {\rm dom}\ g \setminus {\rm dom}\ f$, there exist $x, y \in {\rm dom}\ g$ with $z = x \sqcap y$ and $Z(z) > \max\{Z(x), Z(y)\}$ satisfying 1 or 2:
\begin{enumerate}
\item $x, y \in {\rm dom}\ f \text{ and } \displaystyle g(z) = \frac{1}{2} (g(x) + g(y))$;
\item $x \sqcap y \neq x \sqcup y \text{ and } g(z) = g(x) + g(y) - g(x \sqcup y)$.
\end{enumerate}
In case 1, it holds that $g(z)$ is half-integral, since both $g(x)$ and $g(y)$ are integral.
In case 2, it also holds that $g(z)$ is half-integral, since $g(x)$, $g(y)$, and $g(x \sqcup y)$ are half-integral by the induction hypothesis.
Hence $g(z)$ is half-integral, as required.
\end{proof}

Finally we establish property 2 in Theorem \ref{thm:algo}.
\begin{prop}\label{prop:opt}
Suppose that $f : [k]^2 \rightarrow \overline{\mathbb{R}}$ has a $k$-submodular relaxation. Let $g$ be a $k$-submodular relaxation of $f$ constructed by Algorithm \ref{algo}. For every $k$-submodular relaxation $g^\prime$ of $f$, it holds that
\[ g(z) \geq g^\prime(z) \quad (z \in {\rm dom}\ g).\]
\end{prop}
\begin{proof}
Let $\mathcal{G}$ be a set of $k$-submodular relaxations of $f$. By Algorithm \ref{algo}, for any $z \in {\rm dom}\ g \setminus {\rm dom}\ f$, there exist three cases of $x, y \in [0, k]^2$ with $z = x \sqcap y$ in defining $g(z)$ as follows:
\begin{description}
\item[Case 1:] $Z(x) = Z(y) = 0$,
\item[Case 2:] $Z(x) = Z(y) = 1$,
\item[Case 3:] $Z(x) = 1, Z(y) = 0$.
\end{description}
By Proposition \ref{prop:half}, we do not have to consider the case of $g(z) = (g(x) + g(y))/2$ with $x \in {\rm dom}\ g \setminus {\rm dom}\ f$. So we can represent the three cases as follows:
\begin{description}
\item[Case $1^\prime$:] $\displaystyle g(z) = \frac{1}{2} (g(x_1, x_2) + g(y_1, y_2))$ for some $x_1, x_2, y_1, y_2 \in [k]$,
\item[Case $2^\prime$:] $g(z) = g(0, 0) = g(x, 0) + g(0, y) - g(x, y)$ for some $x, y \in [k]$,
\item[Case $3^\prime$:] $g(z) = g(0, 0) = g(x_1, 0) + g(y_1, y_2) - g(0, y_2)$ for some $x_1, y_1, y_2 \in [k]$ with $x_1 \neq y_1$.
\end{description}
It suffices to prove that for all $g^\prime \in \mathcal{G}$, we have $g(z) \geq g^\prime(z)$ for all cases. Note that for $x \in {\rm dom}\ f = {\rm dom}\ g \cap [k]^2 = {\rm dom}\ g^\prime \cap [k]^2$, it holds that $g(x) = g^\prime(x)$. Furthermore it holds that ${\rm dom}\ g^\prime \supseteq C_{\sqcap, \sqcup}({\rm dom}\ f) = {\rm dom}\ g$, since ${\rm dom}\ g^\prime \supseteq {\rm dom}\ f$ and $g^\prime$ is $k$-submodular.

First we shall consider Case $1^\prime$. For any $g^\prime \in \mathcal{G}$, it holds that $g^\prime(z) \leq (g^\prime(x_1, x_2) + g^\prime(y_1, y_2))/2 = (g(x_1, x_2) + g(y_1, y_2))/2 = g(z)$. So $g^\prime(z) \leq g(z)$, as required.

Next we consider Case $2^\prime$. For any $g^\prime \in \mathcal{G}$, it holds that $g^\prime(0, 0) \leq g^\prime(x, 0) + g^\prime(0, y) - g^\prime(x, y)$. In addition, we have $g^\prime(x, 0) \leq g(x, 0)$ and $g^\prime(0, y) \leq g(0, y)$, since $g(x, 0)$ and $g(0, y)$ are in Case $1'$.
Hence $g^\prime(0, 0) \leq g^\prime(x, 0) + g^\prime(0, y) - g^\prime(x, y) \leq g(x, 0) + g(0, y) - g(x, y) = g(0, 0)$.
So $g^\prime(0, 0) \leq g(0, 0)$, as required.

Finally we consider Case $3^\prime$. We show that this case can reduce to Case $2^\prime$. It holds that $g(x_1, 0) + g(y_1, y_2) - g(0, y_2) \geq g(x_1, 0) + g(0, y_2) - g(x_1, y_2)$ since
\begin{align*}
&(g(x_1, 0) + g(y_1, y_2) - g(0, y_2)) - (g(x_1, 0) + g(0, y_2) - g(x_1, y_2))\\
=\ &g(y_1, y_2) + g(x_1, y_2) - 2g(0, y_2) \geq 0.
\end{align*}
Hence we have
\begin{align*}
g(0, 0) &= g(x_1, 0) + g(y_1, y_2) - g(0, y_2)\\
&\geq g(x_1, 0) + g(0, y_2) - g(x_1, y_2)\\
&\geq g(0, 0).
\end{align*}
So all inequalities are equalities. This means that $g(0, 0) = g(x_1, 0) + g(0, y_2) - g(x_1, y_2)$, and hence Case $3^\prime$ can reduce to Case $2^\prime$, as required.
\end{proof}

We found, by computer experiments, functions on $[k]^n$ for $k, n \geq 3$ or $k = 2, n \geq 4$ such that the maximal relaxation does not exist,
and the proposed algorithm does not output an optimal relaxation $g$, i.e., the minimum value of $g$ is greater than the minimum value of any relaxation.
(We do not know the case for $k = 2, n = 3$.)

\section*{Acknowledgments}
We thank Satoru Fujishige and Kazuo Murota for careful reading and numerous helpful comments, and the referees for helpful comments.
We thank Magnus Wahlstr{\"o}m for discussion on FPT applications.


\end{document}